\documentclass[a4,11pt]{article}

\usepackage{amscd,amsmath,amssymb,amsthm,latexsym}
\usepackage{mathrsfs}
\usepackage{color}
\usepackage{graphicx}

\newcommand{\R}{\mathbb{R}}
\newcommand{\N}{\mathbb{N}}
\newcommand{\C}{\mathbb{C}}
\newcommand{\D}{\mathcal{D}}

\newcommand{\M}{\mathcal{M}}

\newcommand{\e}{{\rm e}}
\newcommand{\supp}{{\rm supp}}

\newtheorem{proposition}{Proposition}[section]

\newtheorem{theorem}[proposition]{Theorem}
\newtheorem{lemma}[proposition]{Lemma}

\theoremstyle{definition}
\newtheorem{definition}[proposition]{Definition}

\numberwithin{equation}{section}

\begin{document}

\pagestyle{myheadings}
\thispagestyle{plain}
	\markboth{AMC, GGG, JAG and SR}{}
		\title{Chaos for generalized Black-Scholes equations}
\author{Anna Maria Candela\thanks{Dipartimento di Matematica, Università degli Studi di Bari Aldo Moro,
			Via E. Orabona 4, 70125 Bari, Italy. 
			Emails: $^1$annamaria.candela@uniba.it,
			$^4$silvia.romanelli@uniba.it} $^{,1}$ \,
 Gisèle Ruiz Goldstein\thanks{Department of Mathematical Sciences, University of Memphis,  
Memphis, TN 38152-3240, USA. 
Emails: $^2$ggoldste@memphis.edu, 
$^3$jgoldste@memphis.edu} $^{,2}$  \\
Jerome A. Goldstein\footnotemark[2] $^{,3}$ \, 
Silvia Romanelli\footnotemark[1] $^{,4}$}
\date{}
	
\maketitle

%--------------------------------------------------

\begin{abstract}
The Nobel Prize winning Black-Scholes equation 
for stock options and the heat equation 
can both be written in the form 
\[
\frac{\partial u}{\partial t}=P_2(A)u,
\] 
where $P_2(z)=\alpha z^2+ \beta z+\gamma$ 
is a quadratic polynomial with $\alpha > 0$. In fact,
taking $A = x\frac{\partial}{\partial x}$ on functions 
on $[0,\infty) \times [0,\infty)$ the previous equality
reduces to the Black-Scholes equation, 
while taking $A = \frac{\partial}{\partial x}$ for functions 
on $\R \times [0,\infty)$
it becomes the heat equation. Here, we ``connect'' the two previous problems 
by considering the generalized operator 
$A= x^a\frac{\partial}{\partial x}$ for functions on $[0,\infty) \times [0,\infty)$ with $0<a<1$, and
our main result is that the corresponding degenerate parabolic equation 
is governed by a semigroup of operators which is chaotic on a class of Banach spaces. 
The relevant Banach spaces are weighted supremum norm spaces 
of continuous functions on $[0,\infty)$.  This paper unifies,
simplifies and significantly extends earlier results 
obtained for the Black-Scholes equation ($a=1$) in \cite{EGG} 
and the heat equation ($a=0$) in \cite{EGG1}.
\end{abstract}

\noindent
{\it \footnotesize 2020 Mathematics Subject Classification}. 
{\scriptsize 47D06, 47A16, 35K05, 35Q91, 91G80}.\\
{\it \footnotesize Key words}. {\scriptsize Hypercyclic semigroup, chaotic semigroup, 
Black-Scholes equation, heat equation, weighted supremum norm space, evolution equation}.

%-----------------------------------

\section{Introduction}

In 2012 H. Emamirad, G. R. Goldstein and J. A. Goldstein in \cite{EGG} 
proved that the Nobel Prize winning stock options research of the 1970s 
led to a second order linear parabolic partial differential equation that is
governed by a one parameter semigroup of bounded linear operators which is
chaotic on certain weighted supremum norm Banach spaces. 
The equation, usually called the Black-Scholes equation, 
originally introduced 
by F. Black, R. C. Merton and M. Scholes in \cite{BS, M}, is
\begin{equation}\label{bms}
\frac{\partial u}{\partial t}=\frac{\sigma^2}{2} x^2 \frac{\partial^2 u}{\partial x^2}
+ r x\frac{\partial u}{\partial x} - ru,
\end{equation}
for $u(x,t)$, a function on $[0, \infty) \times [0, \infty)$. 
This function $u(x,t)$ is the price 
of the stock option when the price of the stock is $x$
and $t$ represents time.
The relevant initial condition in Economics
for \eqref{bms} is 
\[
u(x,t)=(x-p)_{+}
\] 
where $p>0$ is the strike price. Note that this initial condition forces us to focus on unbounded functions.

In \cite{GMR} the authors gave a simple explicit representation 
of the solutions of an abstract Cauchy problem associated with \eqref{bms}

In \cite{EGG} the authors noted that \eqref{bms} can be written as 
\begin{equation} \label{general}
	\frac{\partial u}{\partial t}=P_{2}(A)u
\end{equation} 
where $P_2(z)$ is the real quadratic polynomial 
$P_2(z)=\alpha z^2+\beta z + \gamma$, $\alpha >0$, 
and with $A=x \frac{\partial}{\partial x}$ 
for functions on $[0,\infty) \times [0,\infty)$ 
and appropriate choices of $\alpha$, $\beta$ and $\gamma$
which are related to the volatility and the interest rate.  

After the first systematic study of the notion of chaotic $(C_0)$ semigroup 
in \cite{DSW}, in \cite{EGG1} Emamirad, G. Goldstein and J. Goldstein proved that the heat 
equation was chaotic on certain weighted sup norm spaces 
by taking $A=\frac{\partial}{\partial x}$ in \eqref{general} 
for functions on $\R \times [0,\infty)$.

More recently in \cite{GGK}, 
under suitable assumptions, 
the property of chaos was also faced in the 
framework of a nonautonomous generalized version of \eqref{bms}.

Here we extend the autonomous theory to
\[
	A=\nu x^{a} \frac{\partial}{\partial x}
\]
for functions on $[0,\infty) \times [0,\infty)$ 
with $0 < a <1$, $\nu \ne 0$.  
In this case the corresponding partial differential 
equation \eqref{general} becomes
\[
\frac{\partial u}{\partial t}\ =\nu^{2} x^{2a}\frac{\partial^2 u}{\partial x^2}
+\nu^{2}ax^{2a-1}\frac{\partial u}{\partial x}
+ \beta \nu x^a\frac{\partial u}{\partial x} + \gamma u,
\quad x>0,\, t\ge 0,
\]
with $\alpha =1$ and $\beta$, $\gamma \in \R$. 
To keep the connection with the Black-Scholes problem (corresponding to $a=1$), 
we work on spaces of unbounded functions.

By understanding all underlying principles more clearly, 
we can simplify and clarify the earlier methods. 
Still, our proof is quite complicated and based 
on a lot of heavy machinery developed by others. 
Thus we hope our new and relatively 
selfcontained approach will attract mathematicians 
to use these new tools and advance rigorous mathematical finance even further.

%----------------------------------

\section{Properties of the Underlying Group and Operator}

Let $X$ be a complex Banach space equipped with the supremum norm $\Vert\cdot\Vert$, 
$\R^+ = [0, \infty)$ and let $S(t) \in \mathcal{L}(X)$ 
(the Banach space of all bounded linear operators on $X$) satisfy
\begin{equation}\label{initial}
\left\{\begin{array}{ll}
S(t+s)=S(t)S(s) 
		& \hbox {for all $t, s\in \R^+$},\\
S(0)=I .\\
\end{array}\right.
\end{equation}
Then $S= \{ S(t): t\ge 0 \}$ is an (operator) semigroup. 
A semigroup is \textit{strongly continuous} 
(called a \textit{$(C_0)$ semigroup}) 
if $S(t)f$ is a continuous function of $t \in \R^+$ for all $f \in X$.
If \eqref{initial} and strong continuity hold 
for all $t,s \in \R$, then $S$ is a \textit{$(C_0)$ group}.

For the remainder of the paper,  $\nu>0$ and $0<a<1$ are fixed.
	
We begin by studying the first order Cauchy problem 
\begin{equation}\label{(2.1)}
\left\{\begin{array}{ll}
\frac{\partial u}{\partial t}=\nu x^a \frac{\partial u}{\partial x}
		& \hbox{if $x > 0, \; t \in \R$,}\\
u(x,0)=f(x) &  \hbox{if $x > 0$.}
\end{array}\right. 
	\end{equation} 
If we define the operator $A$ formally by
\begin{equation}\label{operator}
	A \ =\ \nu x^a \frac{\partial}{\partial x},
\end{equation}
then the initial value problem \eqref{(2.1)} can be written in abstract form as
\[
\left\{\begin{array}{ll}
\frac{\partial u}{\partial t} = Au \\
u(0)=f
\end{array}\right.
\] 
where $f$ is a $C^1$ function on $\R^+$.  

Define
\begin{equation}\label{initial2} 
S(t)f(x)\ :=\ f(|x^{1-a}+\nu (1-a)t|^{\frac{1}{1-a}}).
\end{equation} 
The next lemma proves that $S=\{S(t): t \in \R\}$ defined via 
\eqref{initial2} is a group.

\begin{lemma}\label{2-1} 
The unique solution of \eqref{(2.1)} is 
\begin{equation}\label{solution}
u(x,t) = f(|x^{1-a}+\nu (1-a)t|^{\frac{1}{1-a}})=S(t)f(x).
\end{equation}
\end{lemma}

\begin{proof}
Let $x > 0$ and, taking for simplicity $y = x^{1-a} + \nu (1-a)t$, 
from \eqref{solution} and direct computations we have that
\begin{eqnarray}
	\label{partial} 
\frac{\partial u}{\partial t} &=&
\nu f'(|y|^{\frac{1}{1-a}}) |y|^{\frac{a}{1-a}} sgn(y),\\
\nonumber
\frac{\partial u}{\partial x} &=&
f'(|y|^{\frac{1}{1-a}}) |y|^{\frac{a}{1-a}}
sgn(y) x^{-a},
\end{eqnarray}
whence
\begin{equation}
	\label{partial1} 
\nu x^{a} \frac{\partial u}{\partial x} = 
\nu f'(|y|^{\frac{1}{1-a}}) |y|^{\frac{a}{1-a}}
sgn(y).
\end{equation}
Thus, $\frac{\partial u}{\partial t}= \nu x^a \frac{\partial u}{\partial x}$
holds for all $x$, $t$ such that $y \ne 0$. 
Furthermore, by \eqref{partial},
\[
\lim_{y\to 0^+} \frac{\partial u}{\partial t} = 0
\] 
because $f'(0)$ exists and $|y|^{\frac{a}{1-a}} \to 0$
as $y \to 0$ since $0 < a < 1$. Similarly, by \eqref{partial1},
\[
\lim_{y\to 0^+} \left(\nu x^{a} \frac{\partial u}{\partial x}\right) = 0
\] 
and the same limits equal to 0 hold also 
if $y\to 0^-$.
Thus $y=0$ is a removable singularity for
$\frac{\partial u}{\partial t} - \nu x^a \frac{\partial u}{\partial x}$,
so this function can be viewed as a continuous function of $y$
for $y \in \R$, equal to 0 if $y = 0$, 
and so also a continuous function of $(x,t)$ in $\R^* \times \R$.
Uniqueness of the initial value problem gives us that $S$ is a group.
\end{proof}

We now introduce the relevant spaces for our study.  
For any $s\ge 0$, define
\[ 
Y_s\ =\ \left\{f\in C(\R^+): \, 
\lim_{x\to\infty}\frac{|f(x)|}{(1+\e^{s x^{1-a}})(1+\e^{-s x^{1-a}})}\ =\ 0\right\}
\]
equipped with the norm
\[
\Vert f\Vert_{s}\ =\ \sup_{x \in \R^+}
\frac{|f(x)|}{(1+ \e^{s x^{1-a}})(1+\e^{-s x^{1-a}})}\ <\ \infty.
\]

\begin{lemma}\label{2-2}
The operators  $S=\{S(t):t \in \R\}$  defined by \eqref{initial2} 
form a one parameter quasicontractive group on $Y_s$, with
\begin{equation}\label{contractive}
\Vert S (t) f\Vert_{s}\ \le\ 
\e^{\omega |t|}\ \Vert f\Vert_{s} 
\end{equation}
for all $t\in\R$, $f \in Y_s$, 
where $\omega := s\nu (1-a)$.
\end{lemma}

\begin{proof}
First we note that  infinitesimal generator of $S$ is 
formally $A = \nu x^a \frac{d}{d x}$ since
\[
\begin{split}
Af(x) &=\frac{d}{dt}S(t)f(x)|_{t=0}\\
&=\nu x^a f'(x),
\end{split}
\]
Clearly $f \in \D (A)$ requires that both $f$ and $x^af'$ 
are in $Y_s$; we will characterize the domain later. That $S$ is group follows 
from Lemma \ref{2-1}.

To prove the quasicontractivity, we look at
\[
\Vert S(t)f \Vert_s
=\sup_{x>0}\frac{|(S(t)f)(x)|}{(1+ \e^{s x^{1-a}})\ (1+ \e^{-s x^{1-a}})}
=\sup_{x>0}\ \frac{|f(|x^{1-a}+ \nu (1-a)t|^{\frac{1}{1-a}})|}{(1+ \e^{s x^{1-a}})\ (1+\e^{-s x^{1-a}})}.
\]
Set
\[
y\ =\ |x^{1-a} + \nu (1-a)t|^{\frac{1}{1-a}}.
\] 
We have three cases to consider.

{\it Case 1.}  
Let $t \ge 0$.  
Then
\begin{equation}\label{contra2}
y\ =\ (x^{1-a} + \nu (1-a)t)^{\frac{1}{1-a}},
\end{equation}
and so
\[
\e^{sx^{1-a}} = \e^{sy^{(1-a)}} e^{-s \nu (1-a)t}
\] 
and $t\nu >0$.
Then
\[
\begin{split}
	\Vert S(t)f \Vert_s\ &=\ \sup_{y>0} \frac{|f(y)|}{(1+ \e^{s y^{1-a}}\ 
	\e^{-\omega t})\ (1+\e^{-s y^{1-a}}\ \e^{\omega t})}\\
&\le\ \sup_{y>0} \frac{|f(y)|}{ \e^{-\omega t} (\e^{\omega t} + \e^{s y^{1-a}})\ (1+\e^{-s y^{1-a}})}\
\le\ \sup_{y>0} \e^{\omega t} \ \frac{|f(y)|}{(1+ \e^{s y^{1-a}})\ 
	(1+\e^{-s y^{1-a}})}\\
	 &\le\  \e^{\omega t}\ \Vert f\Vert_s
\end{split}
\]
 by our choice of $\omega$, and so the estimate \eqref{contractive} holds.
 
{\it Case 2.}  Let $t < 0$ but $x^{1-a} + \nu (1-a)t \ge 0$.
In this case, $y$ is still given by \eqref{contra2} but $t\nu <0$.
Then we have 
\[
\begin{split}
	\Vert S(t)f \Vert_s\ 
	&\le\  \sup_{y>0} \frac{|f(y)|}{ \e^{\omega t} (1 + \e^{s y^{1-a}})\ 
		(\e^{-\omega t} +\e^{-s y^{1-a}})}\
	\le\  \sup_{y>0} \e^{-\omega t}\ \frac{|f(y)|}{(1+ \e^{s y^{1-a}})\ 
		(1+\e^{-s y^{1-a}})}\\
		&\le\  \e^{\omega |t|}\ \Vert f\Vert_s.
\end{split}
\] 
and so \eqref{contractive} holds.

{\it Case 3.} Let $t < 0$ and $x^{1-a} + \nu (1-a)t \le 0$.
In this case $t\nu <0$ and $y$ is given by 
\[
	y^{1-a} = - x^{1-a} - \nu (1-a)t,
\]
and so
\[
	x^{1-a} = - y^{1-a} -\nu(1-a)t.
\]
Then
\[
\begin{split}
	\Vert S(t)f\Vert_{s} 
	&\le\  \sup_{y>0} \frac{|f(y)|}
	{ \e^{\omega t} (1 + \e^{- s y^{1-a}})\ 
		(\e^{-\omega t} +\e^{s y^{1-a}})}\
	\le\  \sup_{y>0} \e^{-\omega t}\ \frac{|f(y)|}{(1+ \e^{- s y^{1-a}})\ (1+\e^{s y^{1-a}})} \\
	&\le\  \e^{\omega |t|}\ \Vert f\Vert_s.
\end{split}
\]
which is again \eqref{contractive}.
\end{proof}

\begin{lemma} \label{2-3}
$S=\{S(t):t\in\R\}$ is a $(C_0)$ group on $Y_s$,
that is for any $f\in Y_s$ ,
\begin{equation}\label{limit0}
\Vert S(t) f - f\Vert_s \ \to\ 0 
\quad \hbox{as $t \to 0$.}
\end{equation}
\end{lemma}

\begin{proof} 
Because $C^{1}_c(\R^+)$, the space of continuously differentiable 
functions with compact support on $\R^+$, is dense in $Y_s$, 
it is sufficient to prove the statement on $C^{1}_c(\R^+)$.
Note that functions in $C^{1}_c(\R^+)$ need not vanish
at the origin. 
Thus, taking $f\in C^{1}_c(\R^+)$, we have that 
$\supp(f)\subseteq [0,M]$ for some $M< \infty$.
We see that 
\[
\begin{split}
\Vert S(t) f - f \Vert_{s}\ 
&=\sup_{x\in [0,M]}\ \frac{|f(|x^{1-a}+ \nu (1-a)t|^{\frac{1}{1-a}})
-f(x)|}{(1+ \e^{s x^{1-a}})\ (1+\e^{-s x^{1-a}})}\\
&\le \Vert f' \Vert_{\infty} \sup_{x\in [0,M]}\ 
\frac{||x^{1-a}+ \nu (1-a)t|^{\frac{1}{1-a}}-x|}{(1+ \e^{s x^{1-a}})\ (1+\e^{-s x^{1-a}})}\\
& \le C\Vert f' \Vert_{\infty} \sup_{x\in [0,M]}\ \frac{||x^{1-a}
	+ \nu (1-a)t|^{{\frac {1}{1-a}}}-x|}{1+\e^{sx^{1-a}}}.
\end{split}
\]
But the function $|x^{1-a}+ \nu (1-a)t|^{{\frac {1}{1-a}}}$ 
is uniformly continuous on $[0,M] \times [0,1]$, 
so as $t$ approaches $0$,
$||x^{1-a}+\nu (1-a)t|^{\frac {1}{1-a}}-x| \rightarrow 0$ uniformly 
in $t$ for $x\in [0,M]$. Hence, \eqref{limit0} holds.
\end{proof}

In the remainder of this section we present 
some properties of the operator $A= x^a \frac{d}{dx}$.

Notice that the operator $A$ acts on 
$Y_s \cap C_0(0,\infty)$ if we impose the homogeneous Dirichlet 
boundary condition $f(0)=0$, as it is 
\[
C_0(0,\infty) = \{f \in C(\R^+):\ f(0)=0, \,
\lim_{x\to \infty} f(x) = 0\}. 
\]
Also note that while $Y_s$ contains functions that can be unbounded at $\infty$,
functions in $C_0(0,\infty)$ vanish at both 0 and $\infty$.
Furthermore, define
\[
\tilde Y_s = Y_s \cap C_0(\R^+) = \C \oplus (Y_s \cap C_0(0,\infty)).
\] 
Here, the $\oplus$ means direct sum and 
\[
C_0(\R^+) = \{f \in C(\R^+):\ 
\lim_{x\to \infty} f(x) = 0\}.
\]
By definition, $\tilde Y_s$ is a one-dimensional extension 
of $Y_s \cap C_0(0,\infty)$.

We extend $A= x^a \frac{d}{dx}$ to $A_e$ on 
$\tilde Y_s$. Such an extension is given by
\[
A_e(c+f) = 0 + Af = Af\quad 
\hbox{for $f \in Y_s$, $c \in \C$}
\]
(clearly, $x^a \frac{d}{dx}(c) = 0$).
Then $A_e$ generates a $(C_0)$ group on $Y_s$.
No boundary condition is needed at $x=0$. 

Henceforth, we simply write $A$ to mean $A_e$.

We define 
\begin{equation}\label{domain}
\D(A) = \{f\in Y_s\cap C_0^1(0,\infty):\, Af \in Y_s\}.
\end{equation}

\begin{theorem}\label{Thm2-1}
The operator $A$ with domain $\D (A)$ is the generator of the $(C_0)$ quasicontractive group 
$S=\{S(t):t \in \R\}$ on $Y_s$, and hence
\[
R(\pm \lambda I - A) = Y_s
\]
for $\lambda >0$ sufficiently large.
\end{theorem}

\begin{proof}
The statement follows from Lemmas \ref{2-1}, \ref{2-2}, \ref{2-3}
and the Hille-Yosida Theorem.
\end{proof}

We recall the next theorem due to J. A. Goldstein \cite{G1}, 
which is essential to the analysis in the remainder 
of the paper followed by a 
theorem due to Romanov (see \cite[Chapter 2, Section 8]{G2}) 
which also leads to a useful formula in the present setting.

\begin{theorem}\label{Thm2-2}
Suppose that $A$ generates a $(C_0)$ group on a Banach space $X$ 
and let 
\[
p(z)\ =\ (-1)^{n+1} z^{2n} + q(z)
\]
be a polynomial of degree $2n$ where $q(z)$ 
is a polynomial of degree less than $2n$. 
Then $p(A)$ generates an analytic semigroup of 
angle $\frac{\pi}{2}$ on $X$.
\end{theorem}

\begin{theorem}\label{Thm2-3}
Let $A$ be the generator of a $(C_0)$ group $S=\{S(t): t\in \R\}$ 
on a Banach space $X$. Then $A^2$ generates a $(C_0)$ semigroup 
$T=\{T(t): t \ge 0\}$, analytic in the right half plane, given by
\begin{equation}\label{(3)}
	T(t)f\ =\ \dfrac1{\sqrt{4\pi\ t}}\ 
	\int_0^{\infty}
\e^{-\frac{y^2}{4t}}\ [S(y) + S(-y)]f\ dy
\end{equation}
for any $t>0$, $f\in X$.
\end{theorem}

Equation \eqref{(3)} is called \textit{Romanov's formula}.

\begin{theorem}\label{Thm2-4}
Let $A$ be as in \eqref{operator}
with domain $\D(A)$ as in \eqref{domain}. 
Then the operator
\[
A^2 f\ =\ \nu^2 (x^{2a} f''+ a x^{2a-1} f')
\]
with domain
\[
D(A^2)\ =\
\{f\in D(A):\, Af\in D(A),\ 
A^2 f\in Y_s\}
\]
is the generator of the $(C_0)$ semigroup 
$T=\{T(t):t\ge 0\}$, analytic in the right half plane and
given by 
\begin{equation}\label{romanov}
\begin{split}
&T(t)f(x)\ =\ \dfrac1{\sqrt{4\pi\ t}}
\int_0^{\infty}
\e^{-\frac{y^2}{4t}}\ ([S(y) + S(-y)]f)(x)\ dy \\
&\quad =\ \dfrac1{\sqrt{4\pi\ t}}
\int_0^{\infty}
\e^{-\frac{y^2}{4t}}\ [f(|x^{1-a}+\nu(1-a)y|^{\frac{1} {1-a}}) + f(|x^{1-a}-\nu(1-a)y|^{\frac{1} {1-a}})] dy,
\end{split}
\end{equation}
for any $t>0$, $x\in\R^+$ and $f\in Y_s$. 
\end{theorem}

\begin{proof} Generation  for  $A^2$ follows directly from Theorem \ref{Thm2-2}, 
while formula \eqref{romanov} follows from the definition
of $S(t)$ in \eqref{initial2} 
and Romanov's Theorem (here, Theorem \ref{Thm2-3}).
\end{proof}

%-------------------------------

\section{The Chaotic Semigroup  on the Space $Y_s$}

In this section we prove that the semigroup associated with 
the operator $A$ is chaotic.  We begin by defining the notion 
of chaos for semigroups and then introduce
an essential tool, the Godefroy-Shapiro Criterion \cite{GS}, 
for proving chaos in our context.

\begin{definition}
Let $T=\{T(t):t\ge 0\}$ be a $(C_0)$ semigroup 
on a separable infinite dimensional Banach space $X$.
The semigroup $T$ is called \textsl{hypercyclic}
if there exists some $f\in X$ whose orbit $\{T(t)f:\, t\ge 0\}$ is dense in $X$.
\end{definition}

\begin{definition}
A semigroup $T$ 
is called \textsl{chaotic} if it is hypercyclic and 
has a dense set of periodic orbits, that is,
\[
P:=\{f\in X: \, \hbox{some $t_0>0$ exists so that} \,  T(t_0)f=f\}
\]
 is dense in $X$.
\end{definition}

\begin{theorem}(Godefroy-Shapiro Criterion \cite{GS})\label{GSC}
Let $L$ be a bounded linear operator on a  separable infinite dimensional Banach space $X$. 
Let  $Q_1$, $Q_2$ be dense subspaces of $X$ 
and let $Z:Q_1\to Q_2$ be so that the following 
properties hold:
\begin{itemize}
\item[$(i)$] $\; LZy = y$ for all $y\in Q_1$,
\item[$(ii)$] $\; \displaystyle \lim_{n\to\infty}Z^n y = 0$ 
for all $y\in Q_1$,
\item[$(iii)$] $\; \displaystyle
\lim_{n\to\infty}L^n x = 0$ for all $x\in Q_2$.
\end{itemize}
Then $L$ is hypercyclic, that is, 
$\{L^n x: \, n\in\N\}$ is dense in $X$ for some $x\in X$.
\end{theorem}

The following lemma holds by the argument in  
\cite[Lemma 2.2]{EGG} with trivial modifications.

\begin{lemma}
The dual space $(Y_{s})^*$ is 
\[
(Y_s)^*=\{\psi \in  \M_{loc}[0,\infty) : \eta (dx)=
(1+\e^{s x^{1-a}})^{-1}(1+\e^{-s x^{1-a}})^{-1} \psi(dx) \in \M[0,\infty)\}
\]
where $\M[0,\infty)$ is the set of all finite complex Borel measures on $[0,\infty)$. 
\end{lemma}

We are now ready to apply these tools to our setting.

Denote by 
\[
	\mathscr{S}_{s}\
	=\ \{\lambda\in\C: \, 0< Re(\lambda) < s (1-a) \nu \}
\]
the open strip in $\C$ and take
\[
\varphi_{\lambda}(x)\ =\ \e^{\frac{\lambda x^{1-a}}{\nu (1-a)}}, 
\quad x\ge 0.
\]
Then $\varphi_{\lambda}$ is well defined 
and in $C(\R^+)$ for any $\lambda\in \mathscr{S}_{s}$
(in fact for any $\lambda\in\C$).

\begin{lemma}\label{4-2}
 The mapping $\lambda\mapsto \varphi_{\lambda}$ is analytic 
 from the strip $\mathscr{S}_s$ 
 to $Y_s$. 
\end{lemma}

\begin{proof}
Since weak analyticity is equivalent to analyticity, 
it suffices to show weak analyticity, i.e., to show that
\[
\lambda\ \mapsto\ \int_0^\infty \varphi_{\lambda}(x) \ \psi (dx)
\]
is analytic for any $\psi\in F \subseteq \M[0,\infty)$, 
where $F$ is any norm determining subset of the
dual space $(Y_{s})^*$. Here we use 
\[
F\ :=\ \{c\delta_x:\, c\in\C, \, x\in [0,\infty)\},
\]
where $\delta_x$ denotes the Dirac point mass measure at $x$.
Note that
\[
\Vert f\Vert_s\ =\ 
\sup\big\{|c f(x)|=|\langle f,\psi \rangle|: \, 
\psi = c\delta_x, \, c\in\C, \, x\in [0,\infty), 
\, \Vert\psi\Vert_{(Y_{s})^*}=1\big\}.
\]
A choice of $c$ that works is
\[
c\ =\ \frac{1}{(1+\e^{sx_*^{1-a}})\ (1+ \e^{-s x_*^{1-a}})},
\]
where $x_*$ is a supremum point for $f$, i.e.,
\[
\Vert f\Vert_{s}\ =\ |c f(x_*)|.
\] 
Since 
$\varphi_{\lambda}(x) = \langle\varphi_{\lambda}, \delta_x\rangle$
for any $x \in \R^+$, we have that
\[
\lambda\ \mapsto\ 
\e^{\frac{\lambda}{\nu (1-a)} x^{1-a}}\
=\ \langle\varphi_{\lambda},\delta_x\rangle
\]
is an entire function of $\lambda\in\C$ for all $x \ge 0$.
For $\lambda\in \mathscr{S}_s$, we have that
\[
\begin{split}
\lim_{x\to\infty}\frac{|\varphi_\lambda(x)|}
{(1+\e^{s x^{1-a}})\ (1+\e^{-s x^{1-a}})}\ 
&=\ 
\lim_{x\to\infty}\frac{\e^{\frac{Re(\lambda)}{\nu (1-a)} x^{1-a}}}
{(1+\e^{s x^{1-a}})\ (1+\e^{-s x^{1-a}})}\\
&=\ 
\lim_{x\to\infty} \e^{(\frac{Re(\lambda)}{\nu (1-a)} - s) x^{1-a}}\ 
=\ 0
\end{split}
\]
since $Re(\lambda) <s (1-a)\nu$.
Hence $\varphi_{\lambda}\in Y_{s}$.  
Here we used the obvious fact that if $\varphi_{\lambda}\in Y_{s}$, 
then $\varphi_{\lambda}\in Y_{\tau}$ for all $\tau\ge s$.
\end{proof} 

Consider the polynomial 
\begin{equation}\label{poly}
p(z) = z^2 + \beta z + \gamma
\end{equation} for any given $\beta$, $\gamma \in\R$.
The  operator $A$ 
generates $S ={ \{S(t): t\in \R}\}$, a $(C_0)$  
quasicontractive group on $Y_s$. We define 
the  operator
\[
B := p(A) = A^2 + \beta A + \gamma I. 
\]
By Theorem \ref{Thm2-2}, $B$ generates 
an analytic semigroup of angle $\frac{\pi}{2}$ on $Y_s$. 
The semigroup $T=\{T(t):t\in \C, Re(t)>0\} \cup \{0\}$ 
generated by $B$ is given by
\[
T(t)\ :=\ h_t(A),
\]
where 
\begin{equation}\label{exp}
h_t(z) = \e^{t p(z)}, \quad z \in \C,\; Re(t) \ge 0,
\end{equation}
with $p(z)$ as in \eqref{poly}.
 	
Our goal is to prove  $T$ is chaotic on $Y_s$ for some $s>0$.

\begin{lemma}\label{4-3} 
Suppose that there exists an open connected set 
$\Omega\subset \mathscr{S}_s$ 
which has an accumulation point in $\mathscr{S}_s$.
Then the subspace
\[
Q\ =\ {\rm span}\{\varphi_{\lambda}: \ \lambda\in\Omega\}
\]
is dense in $Y_s$.
\end{lemma}

\begin{proof} 
Suppose $\psi\in Q^{\perp}$. 
Since $\varphi_{\lambda}\in Y_s$ and $\psi\in (Y_s)^*$, 
it follows that 
\[
q(\lambda)\ =\ \langle\varphi_{\lambda},\psi\rangle
\]
is well defined and analytic in $\mathscr{S}_s$. 
But $q(\lambda)=0$ for any $\lambda\in\Omega$ 
and $\Omega$ has an accumulation point  of such points $\lambda$ in $\mathscr{S}_s$. 
Hence, $q(\lambda)=0$ in all of $\mathscr{S}_s$ 
by the analytic continuation theorem, and so $\psi \equiv 0$.
This proves the assertion.
\end{proof}

\begin{lemma}\label{4-4} 
Let $\mathbb T$ be the unit circle in $\C$ and consider 
the polynomial $p(z)$ in \eqref{poly} with $\beta$, $\gamma \in \R$,
and $h_t(z)$ as in \eqref{exp} for any fixed $t >0$.
Then
\begin{equation}\label{(3.7)}
h_t(\mathscr{S}_s)\cap \mathbb{T}\neq\emptyset
\end{equation}
and $h_t(\mathscr{S}_{s})\cap \mathbb T$ 
has infinitely many accumulation points in $\mathscr{S}_s$ for $s\ >\ s^*$ where
\begin{equation}\label{critical}
\quad s^* \ :=\ \left\{ 
\begin{array}{ll}
	\frac{-\beta + \sqrt{\beta^2-4\gamma}}{2\nu(1-a)} 
	&\hbox{if $\gamma < \frac{\beta^2}{4}$ and 
		$\ \sqrt{\beta^2-4\gamma} \ge \beta$,}\\
0 &\hbox{otherwise.}
\end{array}
\right.
\end{equation}
\end{lemma}

\begin{proof} 
Showing \eqref{(3.7)} is equivalent to finding 
$z\in \mathscr{S}_{s}$  with
$|\e^{t \ p(z)}|=1$, 
that is,  finding $z\in \mathscr{S}_s$ such that
$Re(p(z))=0$. If $z = x+i y$ with $x = Re(z)$ and $y= Im(z)$, 
from \eqref{(3.7)} we need
\[
0=\ Re(p(z))\ =\ Re(x^2- y^2 + 2i x  y+ \beta (x+i y) + \gamma)\ 
=\ x^2 - y^2 + \beta x + \gamma.
\]
So to prove \eqref{(3.7)} we need to find a $(x,y_0) \in \R^2$
such that 
\[
 \left\{ 
\begin{array}{l}
{y_0}^2 = x^2 + \beta x + \gamma\\
$with$\  0 < x < s (1-a) \nu,
\end{array}
\right.
\]
in other words, we must find a pair $(x,y)$ with $0<x<s\nu (1-a)$ and $y>0$. 
To that end we must investigate when the
graph of parabola 
\[
y = x^2 + \beta x + \gamma
\]
intersects the portion of the strip 
$\mathscr{S}_s$ above the real axis in $\C$.  

Since the parabola opens upward, if $\beta ^{2}-4\gamma <0$, 
the parabola never crosses the real axis and so the 
parabola will always have a segment that intersects 
$\mathscr{S}_s$ above the real axis.  

If  $\beta ^{2} =4\gamma$, the vertex of the parabola
lies on the real axis and hence a portion of the parabola must intersect our strip in the positive half-plane. 

In the case  $\beta ^{2}-4\gamma >0$, the parabola intersects 
the real axis at 
\[
x_{\pm}\ :=\ \frac{- \beta \pm \sqrt{\beta^2 - 4\gamma}}{2}.
\]
Clearly $ x_{-} \le x_+$.
\begin{center}
	\includegraphics[height=4cm]{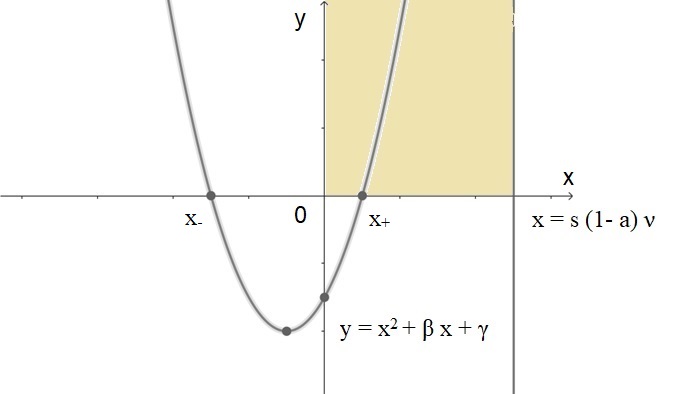}
\end{center}
If $x_{-}>0$, there is a portion intersecting $\mathscr{S}_s$ in the upper half plane.  
If $x_{-} \le0$, we must choose $s$ so large 
that 
\[
x_{+}<s(1-a)\nu.
\]
that is, choose
\[
s\ >\ \frac{- \beta + \sqrt{\beta^2 - 4\gamma}}{2 \nu (1-a)}.
\]
With that choice of $s$, a portion of the parabola intersects $\mathscr{S}_s$ 
in the upper half plane.

The proof of the lemma is now complete by noting once a portion of the parabola 
is in $\mathscr{S}_s \cap \mathbb{T}$,   $h_t(\mathscr{S}_s)\cap \mathbb T$ has 
infinitely many accumulation points.
\end{proof}

We are now ready to state and prove our main result.

\begin{theorem}\label{Thm4-1}
Let $\beta$, $\gamma \in \R$, $B=A^2+\beta A+\gamma I$ where $A=\nu x^a \frac{d}{dx}.$
Then the semigroup generated by to the operator $B$, 
 $T = \{T(t): t\ge 0\}$, governing the problem
 \[
\frac{\partial u}{\partial t}\ =\nu^{2} x^{2a}\frac{\partial^2 u}{\partial x^2}+
\nu^{2}ax^{2a-1}\frac{\partial u}{\partial x}
+ \beta \nu x^a\frac{\partial u}{\partial x} + \gamma u,
\quad x>0,\, t\ge 0,
\] 
is chaotic on $Y_{s}$ if $ s > s^*$, with $s^*$ as in \eqref{critical}.
\end{theorem}

\begin{proof} Recall that
\[
T(t)=h_{t}(A)=\e^{tp(A)}
\]
where $p(z)$ is the polynomial in \eqref{poly}.

The proof is organized in two steps.

\emph{Step 1: $T$ is hypercyclic.} Let $s > s^*$ and fix any $t > 0$. Define the sets
\[
\begin{split}
\Omega_1 & =\ \{\lambda\in \mathscr{S}_{s}:\, 
|h_t(\lambda)|>1\},\\
\Omega_2 &=\ \{\lambda\in \mathscr{S}_{s}:\, 
|h_t(\lambda)|<1\}.	
\end{split}
\]
Consider the set 
\[
Q_j\ =\ {\rm span} \{\varphi_{\lambda} \in Y_s : \, \lambda\in\Omega_j\}, \quad j \in \{1,2\}.
\]
Let $z_0 \in h_t(\mathscr{S}_s)\cap \mathbb T$.  
Since $h_t$ is analytic and nonconstant, $h_t(\mathscr{S}_s)$ is an open set.
Then 
\[
\begin{split}
\Omega_1 & =\ h_t(\mathscr{S}_{s}) \cap  
	\{z\in\C:\, |z|>1\},\\
	\Omega_2 &=\ h_t(\mathscr{S}_{s}) \cap  
	\{z\in\C:\, |z| < 1\}	
\end{split}
\] 
are both open connected sets also. Thus any point in $\Omega_j$ is an accumulation point for $j=1,2.$ 
By Lemma \ref{4-3} it follows that $Q_j$ for $j=1,2$ is dense.  

We want to apply the Godefroy-Shapiro Criterion (here Theorem \ref{GSC})
with $L=h_{t}(A)$. 
Define 
\[
Z=h_{t}(A)^{-1}
\]
on $Q_1$.
Then
\[
Z(\sum_{j=1}^N\alpha_j\varphi_{\lambda_j})=\sum _{j=1}^N\alpha_j(h_{t}(\lambda_j))^{-1} \varphi_{\lambda_j},
\]
for $\lambda_j\in\Omega_j$, $\alpha_j\in\C$ and $N\in\N.$  
It is clear that for any 
$y= \sum _{j=1}^N\alpha_j \varphi_{\lambda_j} \in Q_1$ , we have
$LZy=y$. Furthermore, for 
$\lambda_j\in\Omega_1$, $|h_t(\lambda_j)|>1$ and  so
\[
\lim_{n\to\infty}Z^ny=\lim_{n\to\infty}\sum _{j=1}^N\alpha_j(h_{t}(\lambda_j))^{-n}\varphi_{\lambda_j}=0.
\]
For $y=\sum_{j=1}^N\alpha_j\varphi_{\lambda_j}\in Q_2$, we have $|h_{t}(\lambda_j)|<1$ for each $j$ and so
\[
\lim_{n\to\infty}L^n y=\lim_{n\to\infty}\sum _{j=1}^N\alpha_j(h_{t}(\lambda_j))^{n}\varphi_{\lambda_j}=0.
\]
The hypotheses of Godefroy-Shapiro are then satisfied and so $L$, that is, $T(t)$ 
for each $t$ with $Re(t) >0$, is hypercyclic.

\emph{Step 2: $T$ is chaotic.} 
To see that $L=h_{t}(A)$ is chaotic (and of course recalling we fix any $t>0$),
set
\[
\Omega_3=\{\lambda\in \mathscr{S}_{s}:\, h_{t}(\lambda)\in e^{2\pi i\mathbb Q}\},
\] 
and
\[
Q_3=span\{\varphi_{\lambda}:\, \lambda\in\Omega_3\}.
\]
Then $Q_3$ is contained in the span of all periodic orbits of $L$ with  
period a rational multiple of $2\pi$. Let $\lambda _{j} \in Q_3$; suppose
\[
h_t( \lambda_j)=e^{2\pi i \frac{n_j}{m_j}}.
\]
Then for $y=\sum_{j=1}^N\alpha_j\varphi_{\lambda_j}$, 
and $m=\Pi_{j=1}^N m_j,$ we clearly have
\[
h_{t}(A)^m y=\sum_{j=1}^N[e^{2\pi i n_j \widehat{m}_j}]\alpha_j \varphi_{\lambda_j}=y
\]
(where $\widehat{m}_j=(\Pi_{k\neq j} m_k)$).
So the set of all periodic orbits of $h_{t}(A)$ is dense. Since $t>0$ was fixed 
but otherwise arbitrary, we have shown the semigroup $T$ is chaotic.
\end{proof}

%----------------

\section*{Acknowledgements}
The author Anna Maria Candela acknowledges the partial support of
the following research funds:
MIUR-PRIN Research Project 2017JPCAPN
“\emph{Qualitative and quantitative aspects of nonlinear PDEs}”, 
MUR-PNRR project code CN00000013 “\emph{National Centre for HPC, Big Data and Quantum Computing}” 
- Spoke 10 “\emph{Quantum Computing}” 
(Mission 4 Component 2 Investment 1.4, the European Union – NextGenerationEU), 
\textsl{Università di Bari Aldo Moro} Horizon Europe Seeds Project S51 
“\emph{STEPS: STEerability and controllability of PDES in Agricultural and Physical models}”. \\
The authors Jerome A. Goldstein and Gisèle Ruiz Goldstein 
acknowledge the partial support of \textsl{Università di Bari Aldo Moro} 
Research Funds "\emph{Visiting Researcher/Professor 2022}".  \\
The author Silvia Romanelli acknowledges the partial support of
the GNAMPA-INdAM Project 2022 
"\emph{Anomalous diffusion and its applications to fractal domains, Physics 
and Mathematical Finance}".

%--------------------------------------------
%
%                    References
%
%--------------------------------------------
%

\end{document}